\documentclass[12pt]{amsart}
\usepackage{latexsym}
\usepackage{amssymb,amsmath}
\usepackage[pdftex]{graphicx}
\usepackage{enumerate}
\usepackage{endnotes}
\usepackage{hyperref}
\usepackage[usenames,dvipsnames]{xcolor}
\usepackage{stackrel}
\usepackage{bbm}
\usepackage{tikz}
\usepackage[margin=1in]{geometry}
\usepackage{hyperref}
\usepackage{listings}
\usepackage{courier}
\usepackage{color}
\usepackage{upgreek}

\usetikzlibrary{arrows,chains,matrix,positioning,scopes}

\makeatletter
\tikzset{join/.code=\tikzset{after node path={%
\ifx\tikzchainprevious\pgfutil@empty\else(\tikzchainprevious)%
edge[every join]#1(\tikzchaincurrent)\fi}}}
\makeatother
\tikzset{>=stealth',every on chain/.append style={join},
         every join/.style={->}}
\tikzstyle{labeled}=[execute at begin node=$\scriptstyle,
   execute at end node=$]
\usetikzlibrary{patterns}

\usetikzlibrary{decorations.pathreplacing}

\DeclareSymbolFont{bbold}{U}{bbold}{m}{n}
\DeclareSymbolFontAlphabet{\mathbbold}{bbold}

\newtheorem{thm}{Theorem}[section]
\newtheorem{ithm}{Theorem}
\newtheorem{lem}[thm]{Lemma}
\newtheorem{conj}[thm]{Conjecture}
\newtheorem{prop}[thm]{Proposition}

\theoremstyle{definition}

\theoremstyle{remark}
\newtheorem*{rem}{Remark}

\def\F{{\mathbb F}}

\def\N{{\mathcal N}}
\def\O{{\mathcal O}}
\def\0{{\mathbb 0}}
\def\P{{{\mathbb P}}}
\def\Q{{\mathbb Q}}

\def\Z{{\mathbb Z}}

\newcommand{\lcm}{\text{lcm}}

\newcommand{\GL}{\text{GL}}

\newcommand{\rad}{{\rm{rad}}}

\renewcommand{\bar}{\overline}
\newcommand{\notdiv}{\nmid}

\newcommand{\bfrac}[2]{\left| \frac{#1}{#2} \right|}
\newcommand{\Ell}{\rm{Ell}}
\newcommand{\AV}{\rm{AV}}
\newcommand{\Gal}{\text{Gal}}
\newcommand{\ord}{{\rm{ord}}}

\DeclareSymbolFont{bbold}{U}{bbold}{m}{n}
\DeclareSymbolFontAlphabet{\mathbbold}{bbold}

\usepackage{color}

\begin{document}

\title{Powers in Lucas Sequences via Galois Representations}
\author{Jesse Silliman}
\address{262 Orchard Road, Paoli, PA 19301 }
\email{jksilliman@gmail.com}
\author{Isabel Vogt}
\address{1328 Mulberry Way, Boca Raton, FL 33486}
\email{ivogt161@gmail.com}

\maketitle

\begin{abstract}

Let $u_n$ be a nondegenerate Lucas sequence.  We generalize the results of Bugeaud, Mignotte, and Siksek \cite{siksek06} to give a systematic approach towards the problem of determining all perfect powers in any particular Lucas sequence.  We then prove a general bound on admissible prime powers in a Lucas sequence assuming the Frey-Mazur conjecture on isomorphic mod $p$ Galois representations of elliptic curves.  %We develop a computationally efficient elementary sieve to conditionally determine all powers in several more example sequences.

\end{abstract}

\section{Introduction}
The Fibonacci sequence, perhaps the simplest linear recurrence sequence, begins as \[\underline{0},\underline{1},\underline{1},2,3,5,\underline{8},13,21,34,55,89,\underline{144},233, \cdots\] The reader will no doubt note that the underlined terms are perfect powers. Although a folklore conjecture for many years, a proof that these underlined terms are in fact the \emph{only} perfect powers in the Fibonacci sequence remained decidedly elusive via classical methods. Various partial results in this direction, based mainly on elementary congruences and Baker's theory of linear forms in logarithms, ruled out specific $p$th powers for small primes $p$. In 2006, this conjecture was finally proven in an impressive paper of Bugeaud, Mignotte, and Siksek \cite{siksek06}. In contrast with previous results, the proof relied upon what has come to be termed the ``modular method", connecting the Diophantine behavior of the Fibonacci sequence to the arithmetic properties of elliptic curves.

We revisit this strategy with the aim of systematically approaching the problem of finding all perfect powers in other linear binary recurrence relations.  In particular, we study natural generalizations of the Fibonacci sequence known as Lucas sequences.  A Lucas sequence is a nondegenerate recurrence sequence defined by 
\begin{equation} u_n = b u_{n-1} + c u_{n-2}, \qquad u_0 = 0, u_1 = 1, \end{equation} for $b$ and $c$ nonzero integers.  The Diophantine problem of interest, of course, is to study integer solutions $(n,y,p)$, $n > 0$, $p$ prime, to \begin{equation}\label{the_eqn}u_n = y^p.\end{equation}
Solutions of the form $0=0^p$ or $1 = 1^p$ are called trivial, as they are perfect $p$th powers for any $p$.   

Our more general results are conditional on the Frey-Mazur Conjecture. But for certain special sequences, we get an unconditional result:

\begin{ithm}\label{explicit_eg_thm}
For the following values of $b$ and $c$:
\begin{equation}\label{examples} (b,c) = (3,-2), (5,-6), (7,-12), (17,-72), (9,-20) \end{equation}
the Lucas sequence $u_n$ has no nontrivial $p$th powers, except $u_2 = 3^2$ in $(9,-20)$.  \end{ithm}

\begin{rem}
Note that for $(b,c) = (3,-2)$, Theorem \ref{explicit_eg_thm} provides an alternative proof to the theorem that there are no perfect powers in the Mersenne numbers $2^n - 1$.  This is a specific case of Catalan's Conjecture on perfect powers differing by $1$, recently proven by Mih\u{a}elescu \cite{mih04}.  
\end{rem}

The proof of this theorem relies upon the modular method, which begins by constructing a Frey elliptic curve for an alleged $p$th power.  By the modularity of elliptic curves and Ribet's level lowering theorem, this is associated to a weight $2$ newform whose level $N$ is independent of the alleged solution.  These specific unconditional examples, as well as the necessary classical and modular techniques, are covered in Sections \ref{classicalresults} and \ref{mod_method}.

The modular method is highly effective for the Diophantine problems encountered in Theorem \ref{explicit_eg_thm} as there are no weight 2 newforms for the levels of interest, allowing a contradiction similar to that in Fermat's Last Theorem.  However, if the space of newforms of level $N$ is empty, then $N \leq 60$.  In fact, the sequences in Theorem \ref{explicit_eg_thm} are the only sequences with $b^2+4c = 1$ associated to such levels, see Proposition \ref{onlyseqs}.  The number of newforms of level $N$ grows roughly linearly, thus as $N$ increases, more clever techniques are necessary in order to show that the newforms of a given level cannot correspond to nontrivial solutions to the Diophantine problem.  This usually involves exploiting congruence conditions placed on Fourier coefficients or the existence of complex multiplication, as in \cite{bennett04}.  If a solution could exist, one might proceed as in \cite{siksek06}, deriving ``local conditions" on the index of the solution, which can then be combined with classical methods.  None of these techniques, though, work in general, and often rely heavily on details of the problem at hand, sometimes combined with nontrivial amounts of computation.  Relying on this alone, the possibility of a systematic treatment seems out of reach.

This leads us to the main goal of our paper, which is to use modular methods combined with the powerful Frey-Mazur Conjecture to obtain general results about perfect powers in Lucas sequences.

It was shown independently by Peth{\H{o}} \cite{petho82} and Shorey and Stewart \cite{shorey83} that there are only finitely many perfect powers in any linear recurrence sequence.  However, an \emph{explicit} bound on $p$ has not been published to the knowledge of the authors.  In Section \ref{genthmproof}, we answer this question by proving the following result, conditional on the Frey-Mazur Conjecture (see Section \ref{fm}).  This conjecture essentially says that for $p$ large enough, elliptic curves with isomorphic mod $p$ Galois representations are isogenous.

\begin{ithm}\label{condbound}
Assume the Frey-Mazur Conjecture.  Let 
\[ \psi(N) = N \cdot \prod_{p|N} \left( 1 + \frac{1}{p} \right)\]
be the Dedekind $\psi$ function.  Consider a solution (n,y,p) to \eqref{the_eqn}, with $b^2+4c >0$. Let \[ N = 2^8  \cdot {\rad}'(c) \cdot \rad'(b^2+4c). \]
where $\rad'(m)$ denotes the product over odd primes dividing $m$. Then 
\[ p \leq \max\left\{17,   \psi(N)^{(\psi(N)/12+1)}, 4\log{|\alpha|} \cdot \max\{30,( N+1)\}  \right\} \]
for $\alpha$ the dominant root (i.e. of greatest magnitude) of the characteristic polynomial $g(z) = z^2 -bz-c$.
\end{ithm}

\begin{rem}
The bound coming from $4\log{|\alpha|} \cdot \max\{30,( N+1)\}$ actually only depends upon the largest prime $q \mid N$.
\end{rem}

\begin{rem}
Specializing to the case of $b$ even and $c = -1$, the result of Theorem \ref{condbound} improves upon an unconditional bound published by Bennett in \cite{bennett05}, relying on the modular method combined with classical algebraic number theory.
\end{rem}

In Section \ref{examples}, we use the bounds on $p$ coming from the proof of Theorem \ref{condbound} to return to our aim of systematically finding perfect powers in Lucas sequences.  In particular,  we develop a completely elementary sieve that, when given a bound for $n$ in terms of $p$ derived from Thue equations, can be used to rule out unknown perfect powers in the sequence.  Although considerably simpler, the sieve appears to run as effectively as that used in \cite{siksek06}. To demonstrate this technique, we prove the following theorem conditionally on the Frey-Mazur Conjecture.

\begin{ithm}\label{cond_examples}
Assuming the Frey-Mazur Conjecture, there are no nontrivial perfect powers in the Lucas sequences $(b,c) = (3,1)$, $(5,1)$, and $(7,1)$.
\end{ithm}

\subsection*{Acknowledgements}  The authors wish to thank Professors Ken Ono and David Zureick-Brown for suggesting the topic and for answering questions.  We would also like to thank Eric Larson and other members of the 2013 Emory University REU for useful discussions, as well as Professors Barry Mazur and Samir Siksek for providing helpful comments on the paper and answering questions.  We are also grateful to the NSF for their support.

\section{Classical Facts and Results}\label{classicalresults}

Let $(b,c) \in \Z \times \Z$ define the integral linear binary recurrence relation
\[ U_{n+2} = b\cdot U_{n+1}+ c\cdot U_n, \]
with characteristic polynomial and roots
\[ g(z) = z^2 - bz - c, \qquad \qquad \alpha, \beta = \frac{b \pm \sqrt{b^2+4c}}{2}.\]
The sequence is nondegenerate if $\alpha/\beta$ is not a root of unity.  Throughout this paper, we will refer to nondegenerate integral linear binary recurrence sequences as simply binary recurrence sequences and assume that $b^2+4c >0$.  In particular, let $u_n$ and $v_n$ denote the companion sequences specified by starting conditions
\[ u_0 = 0, u_1 = 1 \qquad \qquad v_0 = 2, v_1 = b .\]
The sequence $u_n$ is termed a Lucas sequence, and will be our main focus.  The $n$th terms of these sequences are given by 
\begin{equation}\label{binetform} u_n = \frac{\alpha^n - \beta^n}{\alpha - \beta} \qquad \qquad v_n = \alpha^n +\beta^n. \end{equation}
This formula easily implies the following key facts
\begin{equation}\label{fib2} u_{2k} = u_kv_k \end{equation}
\begin{equation}\label{gen_diophan}(\alpha - \beta)^2u_n^2 = v_n^2 - 4(\alpha\beta)^n, \end{equation} the second of which is equivalent to $(b^2+4c)u_n^2 = v_n^2 - 4(-c)^n.$  Both \eqref{binetform} and \eqref{gen_diophan} translate the question of perfect powers in recurrence sequences into a Diophantine context that will be the focus of the remaining sections of the paper.

Using classical methods alone, it is often possible to prove that there are no nontrivial perfect $p$th powers in a Lucas sequence \emph{for a specific fixed small value of $p$}.  For example, let $(b,c)$ be one of the sequences in Theorem \ref{explicit_eg_thm}; recall that $b^2+4c = 1$ and $\alpha,\beta \in \Z$ with $\beta = \alpha-1$. 

\begin{lem}\label{relprime}
Let $(b,c)$ be any binary recurrence sequence such that $b^2+4c=1$.  For $n \geq 1$,  $u_n$, $v_n$, and $c$ are relatively prime.
\end{lem}

\begin{proof}
First note that $b^2 +4c = 1$ forces $b$ and $c$ to be relatively prime and $c$ to be even.  By induction on $n$, $u_n$ and $v_n$ are relatively prime to $c$.  And further as $u_n^2  = v_n^2 - 4(-c)^n$, no primes except perhaps those dividing $c$ can divide $u_n$ and $v_n$, and as such they are pairwise relatively prime. \end{proof}

\begin{lem}\label{smallp}
There are no nontrivial squares or cube terms in any of the examples in Theorem \ref{explicit_eg_thm} except $u_2 = 9$ for the sequence $(9,-20)$.
\end{lem}

\begin{proof}

We first deal with the case of $p=2$; we will derive our contradiction from the relation \eqref{binetform}, which in this case reduces to
\[ \alpha^n - (\alpha-1)^n = z^2.\]
Assume that the index $n$, for which $u_n = z^2$, is odd; in this case we can absorb the sign and have a nontrivial integral solution to
\[ x^n +y^n = z^2. \]
But there are no nontrivial solutions to the $(n,n,2)$ Diophantine equation for $n \geq 4$ by \cite{darmon97}.  We easily check that in our examples, $u_3$ is not a square.  In the case $n$ is even, write $n=2^rk$ for $k$ odd.  Then by \eqref{fib2}, 
\[u_{n} = v_{2^{r-1}k}v_{2^{r-2}k} \cdots u_kv_k.\]  
Further by Lemma \ref{relprime}, we know that that these terms are relatively prime.  Thus $u_n=z^2$ if and only if each of these terms is also a square.  In particular, $v_k = \alpha^k + \beta^k$ must be a square.  By the solution of the $(k,k,2)$ Fermat equation, the only possible squares occur for $k=1$, and $n=2$.  It is easy to verify that $u_2$ is not a square except in $(9,-20)$.  The proof for $p=3$ follows in exactly the same way from reduction to odd index and the solution of $(n,n,3)$ by \cite{darmon97}.
\end{proof}

In addition, in some specific cases it is possible to find all perfect powers in a binary recurrence sequence by classical methods alone.  For example, in \cite{petho92}, Peth{\H{o}} proved that $u_7 = 13^2$ is the only nontrivial perfect power in the Pell sequence $(2,1)$.

\section{The Modular Method and Theorem \ref{explicit_eg_thm}}\label{mod_method}

\subsection{The modular method applied to Lucas sequences}

The modular method is a modern approach to solving certain Diophantine equations.  Most famously, the modular method was the key to the celebrated proof of Fermat's Last Theorem \cite{wiles95}, \cite{taylorwiles95}, as well as the determination of perfect powers in the Fibonacci sequence in \cite{siksek06}.

In general, to a hypothetical solution of a Diophantine equation with exponent $p$ we associate a Frey elliptic curve with coefficients depending on the solution.  To these Frey curves we associate the mod $p$ Galois representation
\begin{equation} \rho_{E,p} \colon G_{\Q} \rightarrow \GL_2(\F_p) \end{equation}
corresponding to the action of $G_\Q$ on the $p$-torsion $E[p]$.  These Frey curves will have minimal discriminant of the form $\Delta_E = C \cdot D^p$ for $C$ depending only upon the Diophantine equation, and $D$ depending only upon the hypothetical solution.  The goal is then to use the following deep theorems of modularity and level lowering to show that the Frey curve is not modular, and thus cannot exist, or to derive local information about a solution that could exist. 

\begin{thm}[Modularity of Elliptic Curves \cite{wiles95}, \cite{taylorwiles95}, \cite{conrad01}]\label{modularity}
Let $E$ be an elliptic curve with conductor $N$.  For any prime $p$, there exists a weight 2 newform $f \in S_2(\Gamma_0(N))$ such that
\[ \rho_{E,p} \simeq \rho_{f,p} \]
for $\rho_{f,p}$ the 2-dimensional mod $p$ Galois representation of $f$.\end{thm}

\begin{thm}[Level Lowering \cite{ribet91}]\label{levellow}
Let $f$ be a newform of level $\ell N$ with absolutely irreducible 2-dimensional mod $p$ Galois representation $\rho_{f,p}$ unramified at $\ell$ if $\ell \neq p$ and finite flat at $\ell = p$.  Then there exists a weight 2 newform $g$ of level $N$ such that
\[ \rho_{f,p} \simeq \rho_{g,p}. \]
\end{thm}

For binary recurrence sequences, we will use the fundamental relation \eqref{gen_diophan} as our Diophantine equation.  In particular, say that $u_n = y^p$, that is
\begin{equation}\label{rel_diophan} (b^2+4c)y^{2p}+4(-c)^n = v_n^2 .\end{equation}
This is a solution to the twisted $(p,p,2)$ generalized Fermat equation
\[ (b^2+4c)X^p +4(-c)^nY^p = Z^2. \]

Using the work of Bennett and Skinner \cite{bennett04}, we associate a Frey curve to a \textit{primitive} solution to \eqref{rel_diophan}.  Here primitive denotes $\gcd((b^2+4c)y^{2p}, 4(-c)^n, v_n^2) = 1$.  For any such recurrence relation with $\gcd(b,c)=1$, we choose Frey curves with discriminant and conductor according to the following:

\begin{lem}\label{freycurves}
Assume $u_n = y^p$ for $p\geq 5$, and $n \geq 7$, and $b,c$ relatively prime.  As a convention let $k = \ord_2(b^2+4c)$,  and $b^2+4c = 2^kD$, and $w_n = \pm v_n$ as necessary in the situation (if no congruence conditions are stated then $w_n = v_n$).  Without loss of generality we may assume that we are in one of the following situations:

\begin{enumerate}[1.]

\item $b^2+4c \equiv 1 \pmod{4}$, and $2 \notdiv y,w_n,c$, and $w_n \equiv -(-c)^n \pmod{4}$.
\vspace{-5pt}
\[ E_1: Y^2 = X^3 + w_nX^2 + (-c)^nX \]
\vspace{-20pt}
\[ \Delta = 2^4(-c)^{2n}(b^2+4c)y^{2p},  \qquad N = 2^{\alpha} \rad(2c \cdot (b^2+4c) \cdot y), \qquad \alpha =  \begin{cases} 1: (-c)^n \equiv -1 \pmod{4}\\ 2 : (-c)^n \equiv 1 \pmod{4} \end{cases} \]

\item $b^2+4c \equiv 1,0 \pmod{4}$, $2|y,w_n$, $2 \notdiv c$, let $w_n = 2\hat{w}_n$, with $\hat{w}_n \equiv 1 \pmod{4}$.
\vspace{-5pt}
\[ E_{2} : Y^2 +XY = X^3 + \frac{\hat{w}_n - 1}{4} X^2 + (b^2+4c)2^{-8}u_nX \]
\vspace{-15pt}
\[ \Delta = 2^{-16}(b^2+4c)^2(-c)^ny^{4p}, \qquad N = \rad( c \cdot (b^2+4c)\cdot y)  \]

\item $b^2+4c \equiv 1 \pmod{4}$, $2|c$, $2 \notdiv y,w_n$, and $w_n \equiv 1 \pmod{4}$
\vspace{-5pt}
\[ E_{3}: Y^2 +XY = X^3 +\frac{w_n-1}{4}X^2 +2^{-4}(-c)^nX \]
\vspace{-15pt}
\[ \Delta = 2^{-8}(-c)^{2n}(b^2+4c)y^{2p} , \qquad N =\rad( c \cdot (b^2+4c)\cdot y)  \]

\item $k = 2$, $2 \notdiv y$, $D \equiv -1 \pmod{4}$
\vspace{-5pt}
\[ E_{4} : Y^2 = X^3 + w_nX^2 +Du_n^2X, \qquad \Delta = 2^{6}D^2(-c)^ny^{2p}, \qquad N = 2^{5}\rad( c \cdot D \cdot y )  \]

\item $k = 2$, $2 \notdiv y$,  $D \equiv 1 \pmod{4}$
\vspace{-5pt}
\[E_{5} : Y^2 = X^3 +w_nX^2 +(-c)^nX, \qquad \Delta = 2^{6}D(-c)^{2n}y^{2p}, \qquad N = 2^{5}\rad( c \cdot D \cdot y )  \]

\item $k = 3$, $2 \notdiv y$, $w_n = 2 \hat{w}_n$, 
\vspace{-5pt}
\[ E_{6} : Y^2 = X^3 + w_nX^2 + 2Du_n^2X, \qquad  \Delta = 2^8D^2(-c)^ny^{4p} , \qquad N = 2^6 \rad( c \cdot 2D \cdot y)  \]

\item $k=4$, $2 \notdiv y$, $w_n = 2 \hat{w}_n$,  with $\hat{w}_n \equiv -D \pmod{4}$
\vspace{-5pt}
\[ E_{7} : Y^2 = X^3 +\hat{w}_nX^2 + D u_n X\]
\vspace{-15pt}
 \[ \Delta = 2^4D^2(-c)^ny^{4p}, \qquad N = 2^\alpha \rad( c \cdot 2D \cdot y) ,  \qquad \alpha =  \begin{cases} 1 & : \ D \equiv -1 \pmod{4}\\ 2 & : \ D \equiv 1 \pmod{4} \end{cases} \]

\item $k = 5,6,7$, $2 \notdiv y$, $w_n = 2 \hat{w}_n$, with $\hat{w}_n \equiv 1 \pmod{4}$
\vspace{-5pt}
\[ E_{8} : Y^2 = X^3 + \hat{w}_nX^2 + 2^{k-4}D u_n X \]
\vspace{-20pt}
\[ \Delta = 2^{2k-4}D^2(-c)^ny^{4p}, \qquad N = 2^\alpha \rad(c \cdot 2D \cdot y ) ,  \qquad \alpha =  \begin{cases} 4 &:  \ k = 5 \\ 2 &:  \ k = 6,7 \end{cases} \]

\item $k\geq 8$, $2 \notdiv y$, $w_n = 2 \hat{w}_n$, with $\hat{w}_n \equiv 1 \pmod{4}$
\vspace{-5pt}
\[ E_{9} : Y^2 + XY = X^3 + \frac{\hat{w}_n-1}{4} X^2 + 2^{k-8}D u_n X \]
\vspace{-20pt}
\[ \Delta = 2^{2k-16}D^2(-c)^ny^{4p}, \qquad N = 2^\alpha \rad(c \cdot 2D \cdot y) ,  \qquad \alpha =  \begin{cases} -1&: \ k = 8 \\ 0 &: \ k > 8 \end{cases} \]

\end{enumerate}
\end{lem}
\begin{proof}
We use the formulas of \cite{bennett04}.  For the ease of the reader, the relevant parameters of that paper are 
\begin{center} 
\begin{tabular}{c | c | c | c | c | c | c | c}
Case & \cite{bennett04} Case & $A$ & $B$ & $C$ & $a$ & $b$ & $c$ \\ \hline \hline
1 & (iii) & $b^2+4c$ & $4(-c)^n$ & 1 & $y^2$ & 1 & $w_n$ \\ \hline
2 & (v) & $(-c)^n$ & $(b^2+4c)2^{2p-2}$ & 1 & 1 & $(y/2)^2$ & $w_n/2$ \\ \hline
3 & (v) & $(b^2+4c)$ & $4(-c)^n$ & 1 & $y^2$ & $1$ & $w_n$ \\ \hline
4 & (i) & $(-c)^n$ & $D$ & 1 & $1$ & $y^2$ & $w_n/2$ \\ \hline
5 & (i) & $D$ & $(-c)^n$ & 1 & $y^2$ & 1 & $w_n/2$ \\ \hline
6 & (ii) & $(-c)^n$ & $2D$ & 1 & 1 & $y^2$ & $w_n/2$ \\ \hline
7 & (iii) & $(-c)^n$ & $2^2D$ & 1 & 1 & $y^2$ & $w_n/2$ \\ \hline
8 & (iv) & $(-c)^n$ & $2^{k-2}D$ & 1 & 1 & $y^2$ & $w_n/2$ \\ \hline
9 & (v) & $(-c)^n$ & $2^{k-2}D$ & 1 & 1 & $y^2$ & $w_n/2$ \\ \hline
\end{tabular}
\end{center}
Note that if $c \neq 1$, then $c \notdiv y$ by induction on the recurrence relation in a similar fashion to Lemma \ref{relprime}; thus the only common factors we need to account for are powers of $2$.
\end{proof}

\begin{rem}\label{not_rel_prime}
In the case where $b$ and $c$ are not relatively prime, let $A \mid b,c$ be the largest common factor of $b,c$.  Then the same analysis can be carried out to find that
\[ N_E = 2^{\gamma} \cdot \rad'(A')^2\cdot\rad'(cy(b^2+4c)/A^2) \qquad \qquad \gamma \leq 8, \]
where $A'$ is the squarefree part of $A$.  This is certainly bounded by  
\[ 2^8 \cdot {\rad}'(c) \cdot {\rad}'((b^2+4c)\cdot y). \]
\end{rem}

In order to derive a contradiction or nontrivial local data, we first prove that the mod $p$ representation is unramified at any odd prime not dividing $C$ in $\Delta = C \cdot D^p$; and that it absolutely irreducible at $p \geq 7$, and in some cases $p=5$.  In this way, we can apply Theorem~\ref{levellow} and level-lower to a level independent of the solution.

\begin{lem}\label{unram}
The mod $p$ Galois representation $\rho_{E,p}$ associated to one of the above Frey curves is unramified at any odd prime $\ell$ not dividing $c \cdot (b^2+4c)$.  Further it is finite flat at $p$ if $p \notdiv 2c\cdot (b^2+4c)$.
\end{lem}
\begin{proof}
By the N\'{e}ron-Ogg-Shafarevich criterion, the mod $p$ Galois representation is unramified outside $pN$.  By the theory of Tate curves, $E({\Q}_\ell^{ur}) \simeq {\Q}_\ell^{ur}{}^\ast / q^{\Z}$ with $E[p]({\Q}_\ell^{ur}) = \langle \zeta_p, q^{1/p} \rangle$, so in particular the ramification comes from primes $\ell$ ramified in $\Q(\zeta_p)$ and those such that $p \notdiv \upnu_\ell(q) = \upnu_\ell(\Delta)$.  If $\ell \mid y$, then as $p \mid \upnu_\ell(\Delta)$, $\rho_{E,p}$ is unramified at $\ell$.  Further as $\Q(\zeta_p) \subset \Q(E[p])$, $\rho_{E,p}$ will be ramified at $p$; however, if $p \mid N$, but $p \notdiv 2c \cdot (b^2+4c)$, then $p \mid y$, and the above argument shows that $\Q(E[p]) / \Q(\zeta_p)$ is unramified at $p$.
\end{proof}

\begin{lem}\label{absirr}
Let $E$ be an elliptic curve with a rational $2$-torsion point.   
\begin{enumerate}[(i)]
\item If there exists a prime of multiplicative reduction for $E$, then $\rho_{E,p}$ is absolutely irreducible for $p \geq 7$.
\item Further if $\Delta_E$ is a square, then $\rho_{E,5}$ is absolutely irreducible.
\end{enumerate}
\end{lem}
\begin{proof}
A reducible mod $p$ representation implies that $E$ has a $\Q$-rational $p$ isogeny.  And as $\rho_{E,p}$ is odd ($\det \colon G_\Q \rightarrow \F_p^*$ is the cyclotomic character $\chi_p$, and $\chi_p(c)=-1$ for $c$ complex conjugation), irreducible implies absolutely irreducible.  
Now, first take $p =11$, or \mbox{$p \geq 17$}.  Then by \cite{mazur78} Cor 4.4, $\rho_{E,p}$ is irreducible.  Further, as $E$ has a rational 2-torsion point, it corresponds to a rational point on the modular curve $X_0(2p)$ parameterizing elliptic curves with cyclic $2p$-isogenies.  But there are no noncuspidal, non-CM rational points for $p = 7,13$.

Now assume that $\Delta_E$ is a rational square.  If $\rho_{E,5}$ were reducible, it would give rise to a rational point on $X_0(5) \simeq \P_1$.  Let $X_{\Delta} \simeq \P_1$ be the degree $2$ cover of $X(1)$ parameterizing an elliptic curve $E'$ and a choice of square root of its discriminant $\Delta_{E'}$.  The map to $X(1)$ is given by
\[z \mapsto z^2 + 12^3 = j(E'). \]
As $\Delta_E$ is a rational square, $E$ gives rise to a rational point on $X_{\Delta}$.
As in \cite{brown12} we construct the following diagram
\begin{center}
\begin{tikzpicture}[scale=2]
\matrix (m) [matrix of math nodes, row sep=3em, column sep=3em]
{ X(10) & & X(5)  \\
 & X & X_0(5)  \\
X(2) & X_{\Delta} & X(1)\\};
\path[->,font=\scriptsize,>=angle 90]
(m-1-1) edge  (m-3-1)
(m-1-1) edge (m-1-3)
(m-1-3) edge  (m-2-3)
(m-1-1) edge[dashed]  (m-2-2)
(m-2-3) edge  (m-3-3)
(m-2-2) edge  (m-2-3)
(m-2-2) edge  (m-3-2)
(m-3-1) edge  (m-3-2)
(m-3-2) edge  (m-3-3);
\end{tikzpicture}
\end{center}
where $X$ is the normalization of the fiber product $X_{\Delta} \times_{X(1)} X_0(5)$, and is an elliptic curve given by the equation
\[X \colon \qquad  y^2 = x^3 + 22x^2 +125x .\]
The map to $X_{\Delta}$ is given by
\[ (x,y) \mapsto \frac{y(x^2-500x -15625)}{x^3}. \]
We have that $X$ is rank $0$, with rational points $(0:0:1)$ and $(0:1:0)$, both mapped to the cusp at $\infty$ under the map to $X_{\Delta}$. Thus our elliptic curve cannot give rise to a rational point on $X_0(5)$, and we conclude that $\rho_{E,5}$ is irreducible and thus absolutely irreducible. 
\end{proof}

\subsection{Proof of Theorem \ref{explicit_eg_thm}}

In this section we use the techniques developed in the previous section to give a proof of our unconditional explicit theorem:

\begin{thm}\label{explicit_eg_thm_inplace}
For the following values of $b$ and $c$:
\begin{equation}\label{examples} (b,c) = (3,-2), (5,-6), (7,-12), (17,-72), (9,-20) \end{equation}
the Lucas sequence $u_n$ has no nontrivial $p$th powers, except $u_2 = 3^2$ in $(9,-20)$.
\end{thm}

\noindent
The following lemma will be useful for proving this theorem in the case $p=5$.

\begin{lem}\label{frey5irr}
Let $E$ be a Frey curve arising from a solution $u_n = y^5$ for $b^2+4c = 1$ and $n\geq 5$, then $\rho_{E,5}$ is absolutely irreducible.
\end{lem}
\begin{proof}
If there exists a solution $u_n = y^5$ for $u_n$ such a sequence, then by Lemma \ref{relprime} there exists a primitive solution to 
\[ X^5 + 4(-c)^n Y^5 = (b^2 + 4c) X^5 + 4(-c)^n Y^5 = Z^2. \]
To this we associate the Frey curve of case 3 in Lemma \ref{freycurves}
\[E: Y^2 + XY = X^3 + \frac{w_n - 1}{4} X^2 + 2^{-4}(-c)^nX \]
The Frey curve $E$ has conductor and minimal discriminant
\[ N_E =\rad(c \cdot y)  \qquad \qquad \Delta_E = 2^{-8} \cdot (-c)^{2n} \cdot y^{10}. \]
As $\Delta_E$ is a square, Lemma \ref{absirr} implies $\rho_{E,5}$ is absolutely irreducible.
\end{proof}

\begin{proof}[Proof of Theorem \ref{explicit_eg_thm}]
Posit a solution $u_n = y^p$ for $n \geq 5$ and $p \geq 5$.  To this solution, we associate a Frey curve $E$, as in Lemma \ref{freycurves}.  The mod $p$ Galois representation $\rho_{E,p}$ is unramified outside $p\cdot \rad{(c)}$ and finite flat at $p$  by Lemma \ref{unram} (except at $p=5$ in $(b,c) = (9,-20)$ when $p \mid c$).   By Lemma \ref{absirr} and \ref{frey5irr}, $\rho_{E,p}$ is absolutely irreducible for $p\geq 5$.  By the modularity of elliptic curves \ref{modularity} and Ribet's level-lowering Theorem \ref{levellow}, the representation $\rho_{E,p}$ is isomorphic to one arising from a modular form of level $N = \rad(c)$.  For $b,c$ as above, there are no modular forms of level $\rad(c)=2,6,10$.  So there are no perfect $p$th powers if $p \geq 5$ and $n \geq 5$.

Invoking Lemma \ref{smallp} for the cases $p=2,3$ (and checking $n \leq 4$) completes the proof of the theorem.
\end{proof}

In addition, we have the following interesting proposition that these are the \emph{only} such sequences.

\begin{prop}\label{onlyseqs}
The sequences in Theorem \ref{explicit_eg_thm} are the only Lucas sequences with $b^2+4c = 1$ whose Frey curves descend to levels $N$ such that $\dim S_2(\Gamma_0(N))_{\text{new}} = 0$.
\end{prop}
\begin{proof}
The level descended to is $N = \rad(c)$ by Lemma \ref{freycurves}.  The only even, squarefree levels for which the corresponding space of newforms is trivial are $\rad(c) = N = 2,6,10,22$.  Any solution to $b^2+4c = 1$ with $\rad(c)$ in the above list corresponds to an integral point on an elliptic curve of the form
\[ E \colon Y^2 = Ax^3+1, \]
for finitely many possible $A$.  Enumerating the possible elliptic curves and their integral points easily completes the proof. (When $A = 968$, the curve in question is 13068j1 in the Cremona database; the others can be done in about 30 seconds with Sage, see \cite{code})
\end{proof}

\section{Proof of Theorem \ref{condbound}}\label{genthmproof}

In this section we give a proof of the following general theorem \textbf{conditional on the Frey-Mazur Conjecture}.  Let 
\[\psi(N) = N \cdot \prod_{p \mid N} \left( 1 + \frac{1}{p} \right) \]
be the Dedekind $\psi$ function.  Note that this is the index of $\Gamma_0(N)$ in the full modular group $\Gamma$. In addition 
\[ \dim S_2(\Gamma_0(N))_{\text{new}} \leq 1 + \frac{\psi(N)}{12}. \]
(For an exact formula see \cite{stein07}.)

\begin{thm}\label{condbound_inplace}
Assume the Frey-Mazur Conjecture.  Consider a solution (n,y,p) to \eqref{the_eqn}. Let
\[ N = 2^8  \cdot {\rad}'(c) \cdot\rad'(b^2+4c), \]
where $\rad'(m)$ denotes the product over odd primes dividing $m$. Then 
\[ p \leq \max\left\{17,   \psi(N)^{(\psi(N)/12 + 1)}, 4\log{|\alpha|} \cdot \max\{30,( N+1)\}  \right\}. \]
\end{thm}

\noindent
We break this into the following two theorems.

\begin{thm}[]\label{bound_av}
Define ${\AV}(N)$ to be the maximum prime $p \in \Z$ such that for some higher dimensional modular abelian variety of level $N$, the mod $\mathfrak{p}$ Galois representation is isomorphic to the mod $p$ Galois representation of an elliptic curve of possibly greater conductor over $\Q$, for $\mathfrak{p}$ a prime above $p$.  Then
\[{\AV(N)} \leq \psi(N)^{(\psi(N)/12+1)}. \]
\end{thm}

\begin{thm}[]\label{bound_ell}
Let ${\Ell}(b,c)$ be the maximum prime $p \in \Z$ such that any Frey curve for the recurrence relation (b,c) level lowers, mod $p$, to an elliptic curve.  Assuming the Frey-Mazur Conjecture, 
\[{\Ell}(b,c) \leq \max\left\{17, 4\log{|\alpha|} \cdot\max\{30, (N+1)\} \right\}, \]
where $N = N(b,c)$ is as in Theorem \ref{condbound_inplace}.
\end{thm}

\begin{rem}
We \emph{only} need to assume the Frey-Mazur conjecture for 
\[p > \max\left\{17,   \psi(N)^{(\psi(N)/12 + 1)}  \right\}.\]
\end{rem}

We spend the rest of this section proving these theorems. Theorem \ref{condbound_inplace} immediately follows.  Note that the bound for higher dimension abelian varieties is not conditional on the Frey-Mazur Conjecture.

\subsection{Elliptic Curve Case and the Frey-Mazur Conjecture}\label{fm}

In the results that follow, we rely upon the following empirically-supported question of Mazur \cite{mazur78}, now referred to as the Frey-Mazur Conjecture, concerning the possibility of isomorphic Galois representations arising from non-isogenous elliptic curves.

\begin{conj}[Frey-Mazur]\label{FreyMazur}
Let $p > 17$, and $E_1$ and $E_2$ be elliptic curves over $\Q$, with mod $p$ Galois representations $\rho_{E_1,p}$ and $\rho_{E_2,p}$.  If
\[ \rho_{E_1,p} \simeq \rho_{E_2,p} \]
then $E_1$ is isogenous to $E_2$.
\end{conj}

We prove our result conditional on the Frey-Mazur Conjecture. Our strategy is as follows: if the Frey curve $E$ corresponding to a solution $(n,y,p)$, $p > 17$, level lowers mod $p$ to an elliptic curve $F$, then the Frey-Mazur Conjecture implies that $N_E = N_F$. This forces $\rad(y)$ to divide the product of a fixed set of primes, depending only on $(b,c)$. Then, using general results on smooth numbers in recurrence sequences, we obtain a bound on $n$, in terms of $(b,c)$. We can then bound $p$ such that $y^p = u_n$ in terms of $n$.

\begin{thm}[\cite{gyory82}, \cite{gyory81}, \cite{gyory03}]\label{smoothterm}
Let $S$ be the set of all integers whose prime factors lie in some finite set $\{p_1,p_2,...,p_m\}$ with $p_m \geq p_i$ for all $i$.  Let $u_n$ be a Lucas sequence; if $u_n \in S$ then
\[ n \leq \max\{30, p_m +1 \}. \]
\end{thm}

We can consequently bound $p$ in terms of $n$:

\begin{lem}\label{boundpintermsn}
For all solutions $u_n = y^p$, 
\[ p \leq 4n \log|\alpha|  \]
where $\alpha$ is the dominant root of the characteristic polynomial for $(b,c)$.
\end{lem}

\begin{proof}
It is clear that
\[p\log{2} \leq \log \bfrac{\alpha^n - \beta^n}{\alpha-\beta}  = \log|\alpha^{n-1}+ \alpha^{n-2}\beta+...+\beta^{n-1}|. \]
Thus as $|\alpha| \geq (1 + \sqrt{5})/2$ 
\begin{align*}
p & \leq  \frac{1}{\log{2}} \cdot \log|n\alpha^{n-1}| \\
%& \leq n \log_2|\alpha| +\log_2(n) \\
 & \leq 4\cdot n \log|\alpha|. && \qedhere
\end{align*}
\end{proof}

\begin{proof}[Proof of Theorem \ref{bound_ell}]
Assume there exists a solution $u_n = y^p$, $n > 6$.
We associate a Frey curve $E$ to the Diophantine equation
\[ y^{2p} +4(-c)^n = v_n^2 \]
as in Lemma~\ref{freycurves}.  By Lemma \ref{freycurves}, if $b,c$ are relatively prime, the conductor of $E$ is
\[ N_E = 2^{\gamma}  \cdot \rad'( c \cdot (b^2+4c) \cdot y) \qquad \qquad \gamma \leq 8,\]
Similarly if $A \mid b,c$ is the largest common factor of $b,c$, then
\[ N_E = 2^{\gamma} \cdot \rad'(A')^2\cdot\rad'(cy(b^2+4c)/A^2), \]
where $A'$ is the squarefree part of $A$.
By Lemma~\ref{unram}, the mod $p$ Galois representation $\rho_{E,p}$ is unramified outside $pN_E$, and finite flat at $p \notdiv c(b^2+4c)$. By assumption, there exists an elliptic curve $F$ such that
\[ \rho_{E,p} \simeq \rho_{F,p} \]  
with
\[N_F = 2^{\gamma} \cdot  \rad'(A')^2 \cdot \rad'( c \cdot (b^2+4c) / A^2) \leq 2^8 \cdot \rad'(c) \cdot \rad'(b^2+4c). \]
Invoking the Frey-Mazur Conjecture \ref{FreyMazur}, $E$ and $F$ are isogenous, and thus $N_E = N_F$.  But these differ exactly in the primes dividing $y$, thus
\[ \rad(y) \mid \rad(2c \cdot (b^2+4c)). \]
 An application of Theorem~\ref{smoothterm} and Lemma~\ref{boundpintermsn} concludes the proof.  Note that the tightest bound proved here depends only upon the largest prime factor of $N$.
\end{proof}

\subsection{Higher Dimensional Abelian Variety Case}
We prove a general upper bound on the primes $p$ for which an irrational newform has mod $\mathfrak{p}$ Galois representation isomorphic to that of an elliptic curve of possibly higher conductor, for $\mathfrak{p}$ lying above $p$.
 
Let $E$ be an elliptic curve of conductor $N_E$ and let $a_\ell$ denote the coefficients of the $L$-function of $E$.  If $\rho_{E,p} \simeq \rho_{f,p}$ for $f$ a newform of level $N_f$ with Fourier coeffients $c_\ell\in \O_f$ for $K =  \Q(...,c_\ell,...)$ of degree $n_K = [K:\mathbb{Q}]$, then we have the following important lemma on necessary congruences.

\begin{lem}[\cite{cohen07}]\label{ircong1}
There exists a prime $\mathfrak{p} \mid p$ of $\mathcal{O}_f$ such that, for $\ell$ prime:
\begin{itemize}
\item $c_\ell \equiv a_\ell \mod \mathfrak{p}$, if $\ell \nmid pN_fN_E$
\item $c_\ell^2 \equiv (\ell+1)^2 \mod \mathfrak{p}$, if $\ell \mid\mid N_E, \ell \notdiv pN_f$.
\end{itemize}
Further, as $|a_\ell| < 2\sqrt{\ell}$,
\[p \mid \gcd_{\ell^2 \nmid N}(B(\ell)C(\ell)), \] where
\[B(\ell) = \ell \cdot N_{K / \mathbb{Q}}(c_\ell^2-(\ell+1)^2) \]
\[C(\ell) = \prod_{-2\sqrt{\ell} < r < 2\sqrt{\ell}}{N_{K / \mathbb{Q}}}(c_\ell - r).\]
\end{lem}

For $n_K > 1$, this gives a nontrivial bound on $p$, as there exists an $\ell$ such that $c_\ell \notin \mathbb{Z}$ and $\ell^2 \nmid N$. For such an $\ell$, the product above is nonzero. We can bound $\ell$ using the following well-known theorem:
\begin{thm}[Sturm's Bound \cite{stein07}]\label{sturm}
Let $f,g \in M_k(\Gamma_0(N))$ have Fourier expansions $\sum_n a_nq^n$ and $\sum_n b_n q^n$ respectively.  Then $f = g$ if and only if $a_n = b_n$ for all
\[ n \leq \frac{k}{12} \cdot \psi(N). \]
\end{thm}

\begin{lem}\label{boundell}
Let $f$ be an irrational newform of level $N$.  If $\ell$ is the first prime such that $c_\ell \not\in \Z$, then
\[ \ell \leq (1/6) \cdot \psi(N) .\]
\end{lem}

\begin{proof}
Since $f$ has some irrational coefficient, we know there exists $\sigma \in G_\Q$, such that $f^{\sigma}$ is a distinct newform, also of level $N$.  Then, by \ref{sturm}, they must differ in a coefficient $c_\ell$ for some 
\[ \ell \leq \frac{1}{6} \psi(N). \qedhere\]
\end{proof}

\begin{proof}[Proof of Theorem \ref{bound_av}]
Let $\ell$ be a prime such that $c_\ell \notin \Z$.  If $\ell \mid N$, then $c_\ell = 0, \pm 1 \in \Z$, so we can assume $\ell \notdiv N$. Thus we can apply Lemma \ref{ircong1}, 
\[ p \mid \ell \cdot N_{K / \mathbb{Q}}(c_\ell^2-(\ell+1)^2) \cdot \prod_{-2\sqrt{\ell} < r < 2\sqrt{\ell}}{N_{K / \mathbb{Q}}}(c_\ell - r).\]
It is clear that $|c_\ell^{\sigma}| < 2\sqrt{\ell}$ for all $\sigma \in \Gal(K/\Q)$. Thus, for $k \leq \ell+1$, \[N(c_\ell - k) \leq (\ell+1 + 2\sqrt{\ell})^{n_{K}}.\] 
By Lemma \ref{boundell}, we may take $\ell \leq (1/6)\psi(N)$, and $n_{K} \leq 1+\psi(N)/12$, thus
\[ p \leq \psi(N)^{(\psi(N)/12+1)}. \qedhere\]
\end{proof}

 \subsection{Proof of Theorem \ref{condbound}}
 Let $E$ be a Frey curve, such as in Lemma~\ref{freycurves}.  Then $\rho_{E,p} \simeq \rho_{f,p}$ for a newform $f$ of level bounded by 
 \[ N = 2^8 \cdot {\rad}'(c) \cdot \rad'(b^2+4c). \]
 If $f$ is irrational, we apply Theorem \ref{bound_av} to conclude that $p \leq \psi(N)^{(\psi(N)/12+1)}$.
 If $f$ is rational, then we apply Theorem \ref{bound_ell} to conclude that $p \leq \max\left\{ 17, 4 \log|\alpha| \cdot \max\{30,(N+1)\} \right\}$. 
 The theorem easily follows.

\section{Examples}\label{examples}

In many cases the conditional bounds achieved are in fact much better than the ``worst case" recorded in Theorem \ref{condbound}.   To demonstrate this fact, we find a sharp bound on $p$ for all Lucas sequences with $1 \leq b,c \leq 10$ and relatively prime. For any specific recurrence relation $(b,c)$, we can use the formulas for associating Frey curves in Lemma~\ref{freycurves} to determine the possible levels to which the mod $p$ Galois representation descends, find all possible $p$ such that the representation arises from an irrational newform, and then invoke the Frey-Mazur Conjecture to determine a bound on the index $n$ for which $u_n=y^p$ for any other $p \geq 17$.  Given the reasonable bound on the index from Theorem~\ref{condbound}, it is trivial to check that there are no unknown perfect powers up to that index.  This leaves a finite list of primes $p$ for which there might exist a perfect $p$th power.

\begin{thm} \label{conditional_bound_p}
Let $(b,c)$ identify a Lucas sequence $u_n$ with $1 \leq b,c \leq 10$ and relatively prime.  Then for all solutions $u_n = y^p$, $p \leq 19$. When additionally $c = 1$, we have $p \leq 17$.
\end{thm}
\begin{proof}
Using the formulas for conductors given in Lemma \ref{freycurves}, we compute the list of possible levels to which the Frey curve descends.  Using Lemma \ref{ircong1} we find a sharp bound on $p$ for the possibility of level-lowering to an irrational newform.  Then for $p > 17$, we use Theorem \ref{smoothterm} to bound the index $n$ and check all terms up to this bound.  The Sage transcript of computations can be found at \cite{code}.
\end{proof}

For any prime $p$, given a bound $n \leq B(p)$ we develop a sieve to find all $p$th powers.  This method has some similarity to the sieve in \cite{siksek06}; however, unlike their sieve which uses Frey curves, our method relies entirely upon elementary techniques.  Moreover, our method is completely general, clearly applying to any Lucas sequence and any index $n$ (in contrast, \cite{siksek06} uses the fact that $n$ is prime and $c = 1$).

Although the sieve is completely general, the bound $B(p)$ is sharpest when $c = 1$ and $b^2+4$ is prime.  To demonstrate this method, we prove the following theorem.

\begin{thm}\label{cond_examples}
Assuming the Frey-Mazur Conjecture, there are no nontrivial perfect powers in the Lucas sequences $(3,1)$, $(5,1)$, and $(7,1)$.
\end{thm}

\subsection{The Sieve}

For a given recurrence sequence $(b,c)$ and prime $p$, we want to find the complete list of solutions to $u_n = y^p$.  We derive congruence conditions on $n$ by examining solutions to $u_n = y^p$ over $\F_q$ for $q \equiv 1 \pmod{p}$.  Note that as $p \mid \#\F_q^\times$, the proportion of $p$th powers in our sequence modulo $q$ should heuristically be roughly $1/p$.

For $K(q)$ the period of the Fibonacci sequence mod $q$, let $\N(q) \subseteq \Z/K(q)\Z $ be the set of possible congruences for $n$.  We tighten this set by ``intersecting" data from several primes $S = \{q_i\}$ agreeing via the Chinese Remainder Theorem, to get the set of congruences
\[ \N(S) \subseteq \Z/K(S)\Z, \qquad \qquad K(S)= \lcm(K(q_1),K(q_2),...).\]  
To make this computation manageable,  we choose $q_i$ such that $K(q_i) \mid q_i-1$, and $q_i-1 \mid M$, for $M$ a particularly smooth, increasing modulus.  Although the sieve is completely elementary, it seems to work as efficiently as the sieve in \cite{siksek06}.

For $K(S) > B(p)$, it is clear that the set of congruences are a complete list globally.  The congruences $n \equiv 0,1 \pmod{K(S)}$ will always be present in $\N(S)$ as they are $p$th powers for any $p$.  In addition, if there exists another $p$th power, this will appear as an added congruence in $\N(S)$ that does not grow with the modulus.  All remaining elements of $\N(S)$ should grow with the modulus $M$.  Clearly, we can rule out these possible congruences if the smallest such $\bar{a} > B(p)$.  An implementation of the sieve is available at \cite{code}.

\subsection{Proof of Theorem \ref{cond_examples}}

From Theorem~\ref{conditional_bound_p}, we find that for the sequences in question, $p\leq 17$.  For a solution $(n,y,p)$ to \eqref{gen_diophan}, with can reduce to odd index $n$ via a generalization of an argument of Robbins \cite{robbins83}.  When $n$ is odd, as in \cite{siksek06}, we convert a solution $(n,y,p)$ to \eqref{gen_diophan} into an integral solution $(X,Y)$ of the Thue equation \begin{equation}
\pm 1  =  b \sum_{k=0}^{\lfloor p/2 \rfloor} (-4)^{\frac{p-2k-1}{2}} {p\choose 2k} X^{2k}Y^{p-2k}  + \sum_{k=0}^{\lfloor p/2 \rfloor} (-4)^{\frac{p-2k-1}{2}} {p\choose 2k+1} X^{2k+1}Y^{p-2k-1},
\end{equation} where $y^2 = 4X^2 + Y^2$. We then obtain an upper bound $\max\{|X|, |Y|\} < B$ using the work of Bugeaud and Gy\H{o}ry on Thue equations \cite{bugeaud96}.  Letting $d = b^2+4$, note that $y^p = \left[ \frac{\alpha^n}{\sqrt{d}} \right]$ and it follows that
\begin{equation} n < \frac{\log\left(\sqrt{5^p d}B^p + \sqrt{d}\right)}{\log|\alpha|}. \end{equation}
In particular, for $(b,c) = (3,1), (5,1),(7,1)$, if $p\leq17$,  then \[ n < 10^{800}. \]
Given a bound $n < B(p)$, our sieve heuristically runs in time $O((\log B(p))^{2+\epsilon})$; thus the above bound is computationally attainable.  In fact, using our sieve, we can rule out $p$th powers for indices less than $10^{800}$ in a computation that takes less than 10 hours \cite{code}.

\bibliographystyle{plain}

%\bibliography{bib}{}
%\bibliographystyle{plain}

\end{document}